\documentclass{amsart}
\usepackage{pifont}
\usepackage{amsfonts}

\usepackage{amscd,amssymb,amsmath,graphicx,verbatim}
\usepackage[dvips]{hyperref}
\usepackage[TS1,OT1,T1]{fontenc}

\newtheorem{theorem}{Theorem}[section]
\newtheorem{lemma}[theorem]{Lemma}

\theoremstyle{definition}
\newtheorem{definition}[theorem]{Definition}

\newtheorem{question}[theorem]{Question}

\theoremstyle{remark}
\newtheorem{remark}[theorem]{Remark}

\numberwithin{equation}{section}

\begin{document}

\title{Limits of J-class operators}
\author{Geng Tian}
\address{Geng Tian, Department of Mathematics , Jilin university, 130012, Changchun, P.R.China} \email{tiangeng09@mails.jlu.edu.cn}

\author{Bingzhe Hou}
\address{Bingzhe Hou, Department of Mathematics , Jilin university, 130012, Changchun, P.R.China} \email{houbz@jlu.edu.cn}

\date{Oct. 15, 2010}
\thanks{This work is supported by the National Nature
Science Foundation of China (Grant No. 11001099).}
\begin{abstract}
The purpose of the present work is to answer an open problem which
is raised by G.Costakis and A.Manoussos in their paper "J-class
operators and hypercyclicity " accepted by J. Operator Theory. More
precisely, we give the spectral description of the closure of the
set of J-class operators acting on a separable Hilbert space.
\end{abstract}
\maketitle

\section{Introduction and preliminaries}

In article \cite{G.Costakis and A.Manoussos}, G.Costakis and
A.Manoussos considered  J-class operators and  raised five open
problems. We are interested in  the third problem.
\begin{question}
Is there a spectral description of the closure of the set of J-class
operators acting on a Hilbert space?
\end{question}
In this paper, we shall use the classical approximation tools to
answer this problem.

First, let us recall some definitions and notations about J-class
operators.

Let $X$ be a complex Banach space. Denote by $B(X)$ the set of all
bounded linear operators acting on $X$. Choose any operator $T\in
B(X)$ and any subset $U$ of $X$. The symbol $Orb(T,U)$ denotes the
orbits of $U$ under $T$, i.e. $Orb(T,U)=\{T^nx:~x\in
U,~n=0,1,2,\ldots\}$. If $U=\{x\}$ is a singleton and the orbit
$Orb(T,x)$ is dense in $X$, the operator is called hypercyclic and
the vector $x$ is a hypercyclic vector for $T$.  In paper
\cite{G.Costakis and A.Manoussos}, they somehow "localize" the
notion of hypercyclicity by introducing certain sets, which is
called J-sets. The notion of J-sets is well known in the theory of
topological dynamics, see \cite{BS}. Roughly speaking, if $x$ is a
vector in $X$ and $T$ an operator, then the corresponding J-sets of
$x$ under $T$ describes the asymptotic behavior of all vectors
nearby $x$. To be precise, for a given vector $x\in X$, define
\begin{eqnarray*}J(x)&=&\{y\in X: ~{\rm there~ exist ~a~ strictly~
increasing~ sequence ~of~ positive}\\
&&{\rm integers}~\{k_n\}{\rm ~and~ a ~sequence}~
\{x_n\}\subseteq X {\rm ~such~ that}~x_n\rightarrow x ~{\rm and}\\
&&T^{k_n}(x_n)\rightarrow y\}.
\end{eqnarray*}

\begin{definition}
An operator $T:X\rightarrow X$ is called a J-class operator if there
exists a non-zero vector $x\in X$ so that $J(x)=X$. In this case $x$
will be called a J-class vector for $T$.
\end{definition}

\begin{remark}
If $T$ is a J-class operator, then for any invertible operator $C$,
$C^{-1}TC$ is also a J-class operator.
\end{remark}

In fact it is not difficult to see that if $T$ is hypercyclic then
$T$ is a J-class operator. But the converse implication is not true
in general. See \cite{G.Costakis and A.Manoussos} for examples.

Next we will introduce some notations and properties of Hilbert
space operators. Let $H$ be complex separable Hilbert space and
denote by $B(H)$ the set of bounded linear operators $T:H\rightarrow
H$. For $T\in B(H)$, denote the kernel of $T$ and the range of $T$
by Ker$T$ and Ran$T$ respectively. Denote by $\sigma(T)$,
$\sigma_p(T)$, $\sigma_e(T)$ and $\sigma_{lre}(T)$ the spectrum, the
point spectrum, the essential spectrum and the Wolf spectrum  of $T$
respectively. For $\lambda\in \rho_{s-F}(T):= \mathbb{C}\backslash
\sigma_{lre}(T)$, ${\rm ind}(\lambda-T)={\rm dimKer}(\lambda-T)-{\rm
dimKer}(\lambda-T)^*$, $\mbox{min
ind}(\lambda-T)=\mbox{min}\{\mbox{dimKer}(\lambda-T),
\mbox{dimKer}(\lambda-T)^*\}$. Denote
$\rho_{s-F}^{(n)}(T)=\{\lambda\in\rho_{s-F}(T);~\text{ind}(\lambda-T)=n\}$,
where $-\infty \leq n\leq \infty$,
$\rho_{s-F}^{(+)}(T)=\{\lambda\in\rho_{s-F}(T);~\text{ind}(\lambda-T)>0\}$
and
$\rho_{s-F}^{(-)}(T)=\{\lambda\in\rho_{s-F}(T);~\text{ind}(\lambda-T)<0\}$.
 Denote by $\sigma_0(T)$ the set of isolated
points of $\sigma(T)\backslash\sigma_e(T)$. Denote by $\overline{E}$
the closure of a set $E$.

\section{Closure of the set of J-class operators}

In this section, we will prove our main result. First, let us
introduce some auxiliary lemmas.

\begin{lemma}\label{1}
Let $T\in B(H)$. If $\sigma_p(T^*)\cap\mathbb{D}\neq\emptyset$,
where $\mathbb{D}$ is the open unit disk, then $T$ can not be a
J-class operator.
\end{lemma}
\begin{proof}
If not, suppose $T$ is a J-class operator.

Since $\sigma_p(T^*)\cap\mathbb{D}\neq\emptyset$, choose
$\lambda\in\sigma_p(T^*)\cap\mathbb{D}$, we have ${\rm
dim}ran(T-\overline{\lambda})^\bot={\rm dim}Ker(T^*-\lambda)>0$.

Moreover, we have
$$T-\overline{\lambda}=\begin{matrix}\begin{bmatrix}
\widetilde{T}&A\\
0&0\\
\end{bmatrix}&
\begin{matrix}
\overline{ran(T-\overline{\lambda})}\\
ran(T-\overline{\lambda})^\bot\end{matrix}\end{matrix},$$ i.e.
$$T=\begin{matrix}\begin{bmatrix}
\overline{\lambda}+\widetilde{T}&A\\
0&\overline{\lambda}\\
\end{bmatrix}&
\begin{matrix}
\overline{ran(T-\overline{\lambda})}\\
ran(T-\overline{\lambda})^\bot\end{matrix}\end{matrix}.$$

Because $T$ is a J-class operator, there exists a nonzero vector
$x\oplus y,~x\in\overline{ran(T-\overline{\lambda})},~y\in
ran(T-\overline{\lambda})^\bot$, such that $J(x\oplus y)=H$, i.e.
for any $z\in H$, there exists a strictly increasing sequence of
positive integers $\{k_n\}$ and a sequence $\{x_n\oplus y_n\}$ such
that $\lim\limits_{n\rightarrow\infty}x_n\oplus y_n=x\oplus y$ and
$\lim\limits_{n\rightarrow\infty}T^{k_n}(x_n\oplus y_n)=z$.

Choose a nonzero vector $z\in ran(T-\overline{\lambda})^\bot$, then
from
\begin{eqnarray*}
z\leftarrow T^{k_n}(x_n\oplus
y_n)=[(\overline{\lambda}+\widetilde{T})^{k_n}(x_n)+\sum\limits_{i+j=k_n-1;~i,j\geq0}(\overline{\lambda}+\widetilde{T})^i\overline{\lambda}^jA(y_n)]\oplus\overline{\lambda}^{k_n}y_n,
\end{eqnarray*}
we have
$\lim\limits_{n\rightarrow\infty}\overline{\lambda}^{k_n}y_n=z$.

It follows from $|\lambda|<1$ that ${\rm sup}_n||y_n||=\infty$,
contradict to $\lim\limits_{n\rightarrow\infty}y_n=y$.

Hence $T$ can not be a J-class operator.
\end{proof}

The definition given by Cowen and Douglas \cite{Cowen} is well known
as follows.
\begin{definition}
For $\Omega$ a connected open subset of $\mathbb{C}$ and $n$ a
positive integer, let $B_{n}(\Omega)$ denote the operators $T$ in
$B(H)$ which satisfy:

(1) $\Omega \subseteq \sigma(T)$;

(2) ${\rm ran}(T-\omega)=H \ for \ \omega \ in \ \Omega$;

(3) $\bigvee _{\omega\in \Omega}{\rm ker}(T-\omega)=H$; and

(4) ${\rm dimker}(T-\omega)=n$ for $\omega$ in $\Omega$.
\end{definition}

One often calls the operator $T$ in $B_{n}(\Omega)$ Cowen-Douglas
operator. Denote by $\partial\mathbb{D}$ the boundary of unit open
disk. Then we have the following theorem [\cite{Hbz}, Theorem 3.7].

\begin{theorem}\label{4}
Let $T\in B_n(\Omega)$. If
$\Omega\cap\partial\mathbb{D}\neq\emptyset$, then $T$ is strongly
mixing and hence hypercyclic.
\end{theorem}

\begin{remark}\label{re}
In fact, an extended case of this result, for $n=+\infty$, can be
obtained with the same argument.
\end{remark}

The following lemmas are taken from article \cite{G.Costakis and
A.Manoussos}.

\begin{lemma}\label{2}
Let $X$ be a Banach space and let $Y$ be a separable Banach space.
Consider an operator $S:X\rightarrow X$ so that
$\sigma(S)\subseteq\{\lambda;|\lambda|>1\}$. Let also
$T:Y\rightarrow Y$ be a hypercyclic operator. Then $S\oplus T:
X\oplus Y\rightarrow X\oplus Y$ is a J-class operator but not a
hypercyclic operator.
\end{lemma}

\begin{lemma}\label{5}
For any positive integer $n$, let
$A:\mathbb{C}^n\rightarrow\mathbb{C}^n$ be a linear map. Then $A$ is
not a J-class operator.
\end{lemma}

It is convenient to cite in full length a result of Apostol and
Morrel \cite{HER1}. Let $\Gamma=\partial\Omega$, where $\Omega$ is
an analytic Cauchy domain, and let $L^2(\Gamma)$ be the Hilbert
space of (equivalent classes of) complex functions on $\Gamma$ which
are square integrable with respect to (1/2$\pi$)-times the
arc-length measure on $\Gamma$; $M(\Gamma)$ will stand for the
operator defined as multiplication by $\lambda$ on $L^2(\Gamma)$.
The subspace $H^2(\Gamma)$ spanned by the rational functions with
poles outside $\overline{\Omega}$ is invariant under $M(\Gamma)$. By
$M_+(\Gamma)$ and $M_-(\Gamma)$ we shall denote the restriction of
$M(\Gamma)$ to $H^2(\Gamma)$ and its compression to
$L^2(\Gamma)\ominus H^2(\Gamma)$,respectively, i.e.
$$M(\Gamma)=\begin{matrix}\begin{bmatrix}
M_+(\Gamma)&Z\\
&M_-(\Gamma)\end{bmatrix} &
\begin{matrix}
H^2(\Gamma)\\
H^2(\Gamma)^\bot\end{matrix}\end{matrix}.
$$
\begin{definition}
$S\in B(H)$ is a simple model, if
$$S\simeq\begin{matrix}\begin{bmatrix}
S_+&*&*\\
&A&*\\
&&S_-\\
\end{bmatrix}&
\begin{matrix}
\end{matrix}\end{matrix},$$
where

(1) $\sigma(S_+),~\sigma(S_-),~\sigma(A)$\ are pairwise disjoint;

(2) A is similar to a normal operator with finite spectrum;

(3) $S_+$ is (either absent or) unitarity equivalent to
$\oplus_{i=1}^{m}M_+(\partial\Omega_i)^{(k_i)},~1\leq k_i
\leq\infty$, where $\{\Omega_i\}_{i=1}^{m}$ is a finite family of
analytic Cauchy domains with pairwise diajoint closures;

(4) $S_-$ is (either absent or) unitarity equivalent to
$\oplus_{j=1}^{n}M_-(\partial\Phi_j)^{(h_j)},~1\leq h_j \leq\infty$,
where $\{\Phi_j\}_{j=1}^{n}$ is a finite family of analytic Cauchy
domains with pairwise diajoint closures.
\end{definition}
\begin{theorem}\label{Apostol-Morrel}
The simple models are dense in $B(H)$. More precisely: Given $T\in
B(H)$ and $\epsilon>0$ there exists a simple model $S$ such that

$(1)~ \sigma(S_+)\subseteq \rho_{s-F}^{(-)}(T)\subseteq
\sigma(S_+)_{\epsilon},~\sigma(S_-)\subseteq\rho_{s-F}^{(+)}(T)\subseteq
 \sigma(S_-)_{\epsilon},$  and  $\sigma(A)\subseteq\sigma(T)_{\epsilon},$ where
$(\cdot)_{\epsilon}=\{z\in\mathbb{C};dist[z,\cdot]\leq\epsilon\}.$

$(2)~ {\rm ind}(\lambda-S)={\rm ind}(\lambda-T), \ for \ each \
\lambda\in \rho_{s-F}^{(-)}(S_+)\cup\rho_{s-F}^{(+)}(S_-).$

$(3)~ ||T-S||<\epsilon.$
\end{theorem}

\begin{remark}
The central piece $A$ of the model $S$ can be replaced by many other
operators. For example, let $\Psi$ be an analytic Cauchy domain such
that
$\sigma_{lre}(T)\subseteq\Psi\subseteq\sigma_{lre}(T)_{\epsilon/8}$
and $\Sigma$ an arbitrary perfect subset of $\Psi^-$, then there
exists a simple model $S$, $||T-S||<\epsilon$ such that the Wolf
spectrum of the central piece $A$ is $\Sigma$. One can observe
\cite{HER1} section 6.1 for details.
\end{remark}

It is time to obtain the main result. Let $J(H)$ be the set of
J-class operators on separable Hilbert space $H$. Then we have the
following theorem.

\begin{theorem}
Denote by $E$ the set $\{T\in
B(H);~[\sigma_{lre}(T)\cup\rho_{s-F}^{(+)}(T)]\cap\partial\mathbb{D}\neq\emptyset,~\mathbb{D}^-\cap\rho_{s-F}^{(-)}(T)=\emptyset,~{\rm
and~there~ does~ not~ exist~ a~ component~ of}~\sigma(T)~{\rm
included~ in}~ \mathbb{D}\}$. Then $\overline{J(H)}=E$.
\end{theorem}
\begin{proof}

First, we will show that $\overline{J(H)}\subseteq E$. It follows
from the stability properties of semi-Frodholm operators that $E$ is
closed in norm topology. Hence it suffices to show that
$J(H)\subseteq E$.

Choose any $T\in E^c$, then

\ding {192} there exists a component of $\sigma(T)$ contained in
$\mathbb{D}$, or

\ding {193} there does not exist a component of $\sigma(T)$
contained in $\mathbb{D}$, but there exists a point
$\lambda_0\in\mathbb{D}^-$ such that ${\rm ind}(\lambda_0-T)<0$, or

\ding {194} there does not exist a component of $\sigma(T)$
contained in $\mathbb{D}$ and
$\mathbb{D}^-\cap\rho_{s-F}^{(-)}(T)=\emptyset$, but
$[\sigma_{lre}(T)\cup\rho_{s-F}^{(+)}(T)]\cap\partial\mathbb{D}=\emptyset$.

For the first case, let $\sigma$ be the component of $\sigma(T)$
which contained in $\mathbb{D}$, then from Riesz's decomposition
theorem, we have
\begin{eqnarray*}T&=&\begin{matrix}\begin{bmatrix}
T_1\\
&T_2\\
\end{bmatrix}&
\begin{matrix}
H_1\\
H_2\end{matrix}\end{matrix}\\
&=&\begin{matrix}\begin{bmatrix}
T_1&*\\
&\widetilde{T_2}\\
\end{bmatrix}&
\begin{matrix}
H_1\\
{H_1}^\bot\end{matrix}\end{matrix}\\
&\sim&\begin{matrix}\begin{bmatrix}
T_1\\
&\widetilde{T_2}\\
\end{bmatrix}&
\begin{matrix}
H_1\\
{H_1}^\bot\end{matrix}\end{matrix},\end{eqnarray*} where the symbol
"$\sim$" denotes similarity between operators, $\sigma(T_1)=\sigma$
and
$\sigma(T_2)=\sigma(\widetilde{T_2})=\sigma(T)\setminus\sigma(T_1)$.

Since $r(T_1)<1$ ($r(T_1)$ is the spectrum radius of $T_1$), for
$\epsilon>0$ satisfying $r(T_1)+\epsilon<1$, there exists $N$ such
that $||T_1^n||<(r(T_1)+\epsilon)^n$ for $n\geq N$, i.e.
$\lim\limits_{n\rightarrow\infty}||T_1^n||=0$. Therefore $T$ cannot
be a J-class operator.

For the second case, there must exist a point
$\lambda_0\in\mathbb{D}$ such that ${\rm ind}(\lambda_0-T)<0$. Then
from Lemma \ref{1}, $T$ can not be a J-class operator.

For the third case, either
$\partial\mathbb{D}\subseteq\rho_{s-F}^{(0)}(T)\cap\sigma(T)$, or
there exist only finite isolate normal eigenvalue of $\sigma(T)$ on
$\partial\mathbb{D}$ and $\sigma(T)\cap\mathbb{D}=\emptyset$. Then
from \ref{1} and \ref{5}, we know that $T$ can not be a J-class
operator.

Hence $T\in J(H)^c$. Obviously, $\overline{J(H)}\subseteq E$.

Second, we will show that $E\subseteq \overline{J(H)}$. More
precisely, for any $T\in E$, given any $\epsilon>0$, there exists an
operator $K$ such that $||K||<\epsilon$ and $T+K\in J(H)$.

Choose any $T\in E$. By a standard argument of approximation, there
exists an operator $S\in E$ such that $||T-S||<\delta_1$,
$\sigma(S)$ has only finite components,
$\rho_{s-F}(S)\subseteq\rho_{s-F}(T)$ and ${\rm ind}(\lambda-S)={\rm
ind}(\lambda-T)$ for all $\lambda\in\rho_{s-F}(S)$,
$\sigma_{lre}(S)$ is the closure of an analytic Cauchy domain
$\Omega$, $\rho_{s-F}^{(0)}(S)\cap\sigma(S)$ has only finite points
and $\rho_{s-F}^{(0)}(S)\cap\sigma(S)\subseteq{\mathbb{D}^-}^c$.

Let $\sigma_1$ be the union of the components $\sigma$ of
$\sigma(S)$ which satisfy
$\sigma\cap\partial\mathbb{D}\neq\emptyset$,
$\rho_{s-F}^{(+)}(S)\cap\sigma\cap\partial\mathbb{D}=\emptyset$ and
$\sigma\cap\rho_{s-F}^{(-)}(S)\neq\emptyset$; let $\sigma_2$ be the
union of the components $\sigma$ of $\sigma(S)$ which satisfy
$\rho_{s-F}^{(+)}(S)\cap\sigma\cap\partial\mathbb{D}\neq\emptyset$
and $\sigma\cap\rho_{s-F}^{(-)}(S)\neq\emptyset$; let
$\sigma_3,~\sigma_4,\cdots,\sigma_m$ satisfy
$\sigma_i\cap\partial\mathbb{D}\neq\emptyset$,
$\rho_{s-F}^{(+)}(S)\cap\sigma_i\cap\partial\mathbb{D}=\emptyset$
and $\sigma_i\cap\rho_{s-F}^{(-)}(S)=\emptyset$; let
$\sigma_{m+1},~\sigma_{m+2},\cdots,\sigma_{n}$ satisfy
$\rho_{s-F}^{(+)}(S)\cap\sigma_i\cap\partial\mathbb{D}\neq\emptyset$
and $\sigma_i\cap\rho_{s-F}^{(-)}(S)=\emptyset$; let
$\sigma=\sigma(S)\backslash\bigcup_{i=1}^n\sigma_i$, then
$\sigma\cap\mathbb{D}^-=\emptyset$.

According to Riesz's decomposition theorem, we have
\begin{eqnarray*}S&=&\begin{matrix}\begin{bmatrix}
S_1\\
&S_2\\
&&S_3\\
&&&S_4\\
&&&&S_5\\
\end{bmatrix}&
\begin{matrix}
H(S,\sigma_1)\\
H(S,\sigma_2)\\
H(S,\sigma_{3})\dot{+}
\cdots\dot{+} H(S,\sigma_m)\\
H(S,\sigma_{m+1})\dot{+}
\cdots\dot{+} H(S,\sigma_{n})\\
H(S,\sigma)\end{matrix}\end{matrix},\\S_3&=&
\begin{matrix}\begin{bmatrix}
S_{3,3}\\
&\ddots\\
&&S_{3,m}\\
\end{bmatrix}&
\begin{matrix}
H(S,\sigma_3)\\
\vdots\\H(S,\sigma_m)\end{matrix}\end{matrix},\\~S_4&=&
\begin{matrix}\begin{bmatrix}
S_{4,m+1}\\
&\ddots\\
&&S_{4,n}\\
\end{bmatrix}&
\begin{matrix}
H(S,\sigma_{m+1})\\
\vdots\\
H(S,\sigma_{n})\end{matrix}\end{matrix}.
\end{eqnarray*}

According to theorem \ref{Apostol-Morrel}, there exists $A_1\in
B(H(S,\sigma_1)),~A_2\in B(H(S,\sigma_2))$,
$||A_1||<\delta_2,~||A_2||<\delta_2$ such that $S_1+A_1,~S_2+A_2$
are simple models, $S_1+A_1,~S_2+A_2\in E$, $\sigma_{lre}(S_1+A_1)$,
$\sigma_{lre}(S_2+A_2)$ are the closure of an analytic Cauchy domain
and for any component $\sigma$ of $\sigma(S_1+A_1)$ (or
$\sigma(S_2+A_2)$) satisfying
$\sigma\cap\partial\mathbb{D}\neq\emptyset$,
$\rho_{s-F}^{(-)}(S_1+A_1)\cap\sigma=\emptyset$ (or
$\rho_{s-F}^{(-)}(S_2+A_2)\cap\sigma=\emptyset$).

Let
$\sigma_{1,1},\ldots,\sigma_{1,k},\sigma_{1,k+1},\ldots,\sigma_{1,k+p}$
be the components of $\sigma(S_1+A_1)$ such that
$\sigma_{1,i}\cap\partial\mathbb{D}\neq\emptyset,~1\leq i\leq k+p$,
where
$\rho_{s-F}^{(+)}(S_1+A_1)\cap\sigma_{1,i}\cap\partial\mathbb{D}=\emptyset,~1\leq
i\leq k$,
$\rho_{s-F}^{(+)}(S_1+A_1)\cap\sigma_{1,i}\cap\partial\mathbb{D}\neq\emptyset,~k+1\leq
i\leq k+p$; let
$\sigma_{1,k+p+1}=\sigma(S_1+A_1)\backslash\bigcup_{i=1}^{k+p}\sigma_{1,i}$,
then $\sigma_{1,k+p+1}\cap\mathbb{D}^-=\emptyset$; let
$\sigma_{2,1},\ldots,\sigma_{2,l},\sigma_{2,l+1},\ldots,\sigma_{2,l+q}$
be the components of $\sigma(S_2+A_2)$ such that
$\sigma_{2,i}\cap\partial\mathbb{D}\neq\emptyset,~1\leq i\leq l+q$,
where
$\rho_{s-F}^{(+)}(S_2+A_2)\cap\sigma_{2,i}\cap\partial\mathbb{D}=\emptyset,~1\leq
i\leq l$,
$\rho_{s-F}^{(+)}(S_2+A_2)\cap\sigma_{2,i}\cap\partial\mathbb{D}\neq\emptyset,~l+1\leq
i\leq l+p$; let
$\sigma_{2,l+q+1}=\sigma(S_2+A_2)\backslash\bigcup_{i=1}^{l+q}\sigma_{2,i}$
then $\sigma_{2,l+q+1}\cap\mathbb{D}^-=\emptyset$.

By Riesz's decomposition theorem,
\begin{eqnarray*}S_1+A_1&=&\begin{matrix}\begin{bmatrix}
S_{1,1}\\
&\ddots\\
&&S_{1,k}\\
&&&S_{1,k+1}\\
&&&&\ddots\\
&&&&&S_{1,k+p}\\
&&&&&&S_{1,k+p+1}\\
\end{bmatrix}&
\begin{matrix}
H(S_1+A_1,\sigma_{1,1})\\
\vdots\\
H(S_1+A_1,\sigma_{1,k})\\
H(S_1+A_1,\sigma_{1,k+1})\\
\vdots\\
H(S_1+A_1,\sigma_{1,k+p})\\
H(S_1+A_1,\sigma_{1,k+p+1})\\
\end{matrix}\end{matrix},\\S_2+A_2&=&\begin{matrix}\begin{bmatrix}
S_{2,1}\\
&\ddots\\
&&S_{2,l}\\
&&&S_{2,l+1}\\
&&&&\ddots\\
&&&&&S_{2,l+q}\\
&&&&&&S_{2,l+q+1}\\
\end{bmatrix}&
\begin{matrix}
H(S_2+A_2,\sigma_{2,1})\\
\vdots\\
H(S_2+A_2,\sigma_{2,l})\\
H(S_2+A_2,\sigma_{2,l+1})\\
\vdots\\
H(S_2+A_2,\sigma_{2,l+q})\\
H(S_2+A_2,\sigma_{2,l+q+1})\\
\end{matrix}\end{matrix}.
\end{eqnarray*}

Next we deal with the operators
$S_{1,1},\ldots,S_{1,k},S_{2,1},\ldots,S_{2,l},S_{3,3},\ldots,S_{3,m}$.

Let
$\mathbb{D}_{1,1},\ldots,\mathbb{D}_{1,k},~\mathbb{D}_{2,1},\cdots,\mathbb{D}_{2,l},~\mathbb{D}_{3,3},\cdots,\mathbb{D}_{3,m}$
be a finite family of pairwise disjoint open disks included in
$\sigma_{lre}(S_1+A_1)$, $\sigma_{lre}(S_2+A_2)$ and
$\sigma_{lre}(S_3)$ respectively such that
$\mathbb{D}_{i,j}\cap\partial\mathbb{D}\neq\emptyset$ and
$\mathbb{D}_{1,i}\subseteq\sigma_{1,i}~(i=1,2,\cdots,k)$,
$\mathbb{D}_{2,i}\subseteq\sigma_{2,i}~(i=1,2,\cdots,l)$,
$\mathbb{D}_{3,i}\subseteq\sigma_{i}~(i=3,\cdots,m)$.

By using the Similarity Orbit Theorem \cite{HER}, there exist
operators $C_{1,1},\ldots,C_{1,k}$, $C_{2,1},\ldots,C_{2,l}$,
$C_{3,3},\ldots,C_{3,m}$, $||C_{i,j}||<\delta_3$ such that
$S_{1,i}+C_{1,i}\in E,~1\leq i\leq k$, $S_{2,i}+C_{2,i}\in E,~1\leq
i\leq l$, $S_{3,i}+C_{3,i}\in E,~3\leq i\leq m$ and for $1\leq i\leq
k$
$\rho_{s-F}(S_{1,i}+C_{1,i})=\rho_{s-F}(S_{1,i})\cup\mathbb{D}_{1,i}$,
${\rm ind}(S_{1,i}+C_{1,i}-\lambda)={\rm ind}(S_{1,i}-\lambda)$ for
all $\lambda\in\rho_{s-F}(S_{1,i})$, ${\rm
ind}(S_{1,i}+C_{1,i}-\lambda)=1$ for $\lambda\in\mathbb{D}_{1,i}$;
for $1\leq i\leq l$,
$\rho_{s-F}(S_{2,i}+C_{2,i})=\rho_{s-F}(S_{2,i})\cup\mathbb{D}_{2,i}$,
${\rm ind}(S_{2,i}+C_{2,i}-\lambda)={\rm ind}(S_{2,i}-\lambda)$ for
all $\lambda\in\rho_{s-F}(S_{2,i})$, ${\rm
ind}(S_{2,i}+C_{2,i}-\lambda)=1$ for $\lambda\in\mathbb{D}_{2,i}$;
for $3\leq i\leq m$,
$\rho_{s-F}(S_{3,i}+C_{3,i})=\rho_{s-F}(S_{3,i})\cup\mathbb{D}_{3,i}$,
${\rm ind}(S_{3,i}+C_{3,i}-\lambda)={\rm ind}(S_{3,i}-\lambda)$ for
all $\lambda\in\rho_{s-F}(S_{3,i})$, ${\rm
ind}(S_{3,i}+C_{3,i}-\lambda)=1$ for $\lambda\in\mathbb{D}_{3,i}$.

For $S_{1,1}+C_{1,1},\ldots,S_{1,k}+C_{1,k}$, $S_{1,k+1},\ldots,$
$S_{1,k+p}$, $S_{2,1}+C_{2,1},$ $\ldots,S_{2,l}+C_{2,l}$,
$S_{2,l+1},\ldots,S_{2,l+q}$,
$S_{3,3}+C_{3,3},\ldots,S_{3,m}+C_{3,m},S_{4,m+1},\ldots,S_{4,n}$,
from \cite{Her1} (Corollary 2.4 in it), there exists operators
$D_{1,i},~1\leq i\leq k+p$, $D_{2,i},~1\leq i\leq l+q$,
$D_{3,i},~3\leq i\leq m$, $D_{4,i},~m+1\leq i\leq n$,
$||D_{i,j}||<\delta_4$ such that $S_{1,i}+C_{1,i}+D_{1,i},~~1\leq
i\leq k$, $S_{1,i}+D_{1,i},~k+1\leq i\leq k+p$,
$S_{2,i}+C_{2,i}+D_{2,i},~1\leq i\leq l$, $S_{2,i}+D_{2,i},~l+1\leq
i\leq l+q$, $S_{3,i}+C_{3,i}+D_{3,i},~3\leq i\leq m$,
$S_{4,i}+D_{4,i},~m+1\leq i\leq n$ are all Cowen-Douglas operators
with their domains intersect the unit circle, hence from theorem
\ref{4}, they are all hypercyclic.

Let \begin{eqnarray*} &K&=S-T+\\
&&\begin{matrix}\begin{bmatrix}
A_1+C_1+D_1\\
&A_2+C_2+D_2\\
&&C_3+D_3\\
&&&D_4\\
&&&&0\\
\end{bmatrix}&
\begin{matrix}
H(S,\sigma_1)\\
H(S,\sigma_2)\\
H(S,\sigma_{3})\dot{+}
\cdots\dot{+} H(S,\sigma_m)\\
H(S,\sigma_{m+1})\dot{+}
\cdots\dot{+} H(S,\sigma_{n})\\
H(S,\sigma)\end{matrix}\end{matrix},
\end{eqnarray*} where \begin{eqnarray*} C_1&=&
\begin{matrix}\begin{bmatrix}
C_{1,1}\\
&\ddots\\
&&C_{1,k}\\
&&&0\\
&&&&\ddots\\
&&&&&0\\
&&&&&&0\\
\end{bmatrix}&
\begin{matrix}
H(S_1+A_1,\sigma_{1,1})\\
\vdots\\
H(S_1+A_1,\sigma_{1,k})\\
H(S_1+A_1,\sigma_{1,k+1})\\
\vdots\\
H(S_1+A_1,\sigma_{1,k+p})\\
H(S_1+A_1,\sigma_{1,k+p+1})\\\end{matrix}\end{matrix},\\
 C_2&=&
\begin{matrix}\begin{bmatrix}
C_{2,1}\\
&\ddots\\
&&C_{2,l}\\
&&&0\\
&&&&\ddots\\
&&&&&0\\
&&&&&&0\\
\end{bmatrix}&
\begin{matrix}
H(S_2+A_2,\sigma_{2,1})\\
\vdots\\
H(S_2+A_2,\sigma_{2,l})\\
H(S_2+A_2,\sigma_{2,l+1})\\
\vdots\\
H(S_2+A_2,\sigma_{2,l+q})\\
H(S_2+A_2,\sigma_{2,l+q+1})\\\end{matrix}\end{matrix},\\
C_3&=&
\begin{matrix}\begin{bmatrix}
C_{3,3}\\
&\ddots\\
&&C_{3,m}\\
\end{bmatrix}&
\begin{matrix}
H(S,\sigma_3)\\
\vdots\\H(S,\sigma_m)\end{matrix}\end{matrix},\\
D_1&=&
\begin{matrix}\begin{bmatrix}
D_{1,1}\\
&\ddots\\
&&C_{1,k+p}\\
&&&0\\
\end{bmatrix}&
\begin{matrix}
H(S_1+A_1,\sigma_{1,1})\\
\vdots\\
H(S_1+A_1,\sigma_{1,k+p})\\
H(S_1+A_1,\sigma_{1,k+p+1})\\\end{matrix}\end{matrix},\\
D_2&=&
\begin{matrix}\begin{bmatrix}
D_{2,1}\\
&\ddots\\
&&D_{2,l+q}\\
&&&0\\
\end{bmatrix}&
\begin{matrix}
H(S_2+A_2,\sigma_{2,1})\\
\vdots\\
H(S_2+A_2,\sigma_{2,l+q})\\
H(S_2+A_2,\sigma_{2,l+q+1})\\\end{matrix}\end{matrix},\\
D_3&=&
\begin{matrix}\begin{bmatrix}
D_{3,3}\\
&\ddots\\
&&D_{3,m}\\
\end{bmatrix}&
\begin{matrix}
H(S,\sigma_3)\\
\vdots\\H(S,\sigma_m)\end{matrix}\end{matrix},\\
D_4&=&
\begin{matrix}\begin{bmatrix}
D_{4,m+1}\\
&\ddots\\
&&C_{4,n}\\
\end{bmatrix}&
\begin{matrix}
H(S,\sigma_{m+1})\\
\vdots\\
H(S,\sigma_{n})\end{matrix}\end{matrix},
\end{eqnarray*} then if $\delta_1,\delta_2,\delta_3,\delta_4$ are chosen small enough, we have $||K||<\epsilon$ and from lemma \ref{2}, $T+K\in J(H)$.
Hence $E\subseteq \overline{J(H)}$.
\end{proof}

{\bf Acknowledgements} A large part of this article was developed
during the seminar on operator theory and dynamical system held at
the University of Jilin in China. The authors are deeply indebted to
professor Cao Yang for many inspiring discussions.

\end{document}